\documentclass{amsart}
\usepackage{amsfonts,graphicx,rotating}
\usepackage{amssymb}
\usepackage[hidelinks]{hyperref}
\usepackage{tikz-cd}
\usepackage{ctable}

\makeatletter
\@namedef{subjclassname@2020}{\textup{2020} Mathematics Subject Classification}
\makeatother

\newcommand{\im}{\operatorname{im}}

\newtheorem{theorem}{Theorem}[section]
\newtheorem{lemma}[theorem]{Lemma}

\newtheorem{proposition}[theorem]{Proposition}

\theoremstyle{definition}

\theoremstyle{remark}

\numberwithin{equation}{section}


\begin{document}

\title[On the mod $2$ cohomology algebra of oriented Grassmannians]{On the mod $2$ cohomology algebra of oriented Grassmannians}

\author{Milica Jovanovi\'c}
\address{University of Belgrade,
  Faculty of mathematics,
  Studentski trg 16,
  Belgrade,
  Serbia}
\email{milica.jovanovic@matf.bg.ac.rs}

\author{Branislav I.\ Prvulovi\'c}
\address{University of Belgrade,
  Faculty of mathematics,
  Studentski trg 16,
  Belgrade,
  Serbia}
\email{branislav.prvulovic@matf.bg.ac.rs}

\thanks{The first author was partially supported by the Science Fund of the Republic of Serbia, Grant No.\ 7749891, Graphical Languages -- GWORDS. The second author was partially supported by the Ministry of Science, Technological Development and Innovations of the Republic of Serbia [contract no.\ 451-03-47/2023-01/200104].}





\begin{abstract}
For $n\in\{2^t-3,2^t-2,2^t-1\}$ ($t\ge3$) we study the cohomology algebra $H^*(\widetilde G_{n,3};\mathbb Z_2)$ of the Grassmann manifold $\widetilde G_{n,3}$ of oriented $3$-dimensional subspaces of $\mathbb R^n$. A complete description of $H^*(\widetilde G_{n,3};\mathbb Z_2)$ is given in the cases $n=2^t-3$ and $n=2^t-2$, while in the case $n=2^t-1$ we obtain a description complete up to a coefficient from $\mathbb Z_2$.
\end{abstract}

\maketitle



\section{Introduction and statement of main results}
\label{intro}

Let $G_{n,k}$ be the Grassmann manifold of $k$-dimensional subspaces in $\mathbb R^n$ and let $\widetilde G_{n,k}$ be the Grassmann manifold of \em oriented \em $k$-dimensional subspaces in $\mathbb R^n$. Over both manifolds there are canonical $k$-dimensional vector bundles, denoted by $\gamma_{n,k}$ and $\widetilde\gamma_{n,k}$ respectively. There are several ways to describe the mod $2$ cohomology algebra $H^*(G_{n,k};\mathbb Z_2)$, one of them being due to Borel's famous result \cite{Borel}, which establishes that every element in $H^*(G_{n,k};\mathbb Z_2)$ can be expressed as a polynomial in Stiefel--Whitney classes $w_i(\gamma_{n,k})\in H^i(G_{n,k};\mathbb Z_2)$, $1\le i\le k$. On the other hand, the situation in $H^*(\widetilde G_{n,k};\mathbb Z_2)$ is more complicated, although $\widetilde G_{n,k}$ is simply connected, and so $w_1(\widetilde\gamma_{n,k})=0$. Namely, besides the Stiefel--Whitney classes $w_i(\widetilde\gamma_{n,k})\in H^i(\widetilde G_{n,k};\mathbb Z_2)$, which we from now on abbreviate to $\widetilde w_i$, $2\le i\le k$, there are some additional indecomposable classes (a positive dimensional cohomo\-logy class is \em indecomposable \em if it cannot be written as a polynomial in classes of smaller dimension). Put in other words, the subalgebra of $H^*(\widetilde G_{n,k};\mathbb Z_2)$ generated by $\widetilde w_2,\ldots,\widetilde w_k$ does not coincide with $H^*(\widetilde G_{n,k};\mathbb Z_2)$. There is a description of this subalgebra as a quotient of the polynomial algebra in $\widetilde w_2,\ldots,\widetilde w_k$ (analogous to the Borel description of $H^*(G_{n,k};\mathbb Z_2)$), but the whole algebra $H^*(\widetilde G_{n,k};\mathbb Z_2)$ is not completely described in general.

The case $k=1$ is trivial, since $\widetilde G_{n,1}$ is the $(n-1)$-dimensional sphere $S^{n-1}$. The algebra $H^*(\widetilde G_{n,k};\mathbb Z_2)$ is known in the case $k=2$ as well. In \cite{KorbasRusin:Palermo} one can find a complete description of this algebra for $k=2$. In the same paper, the additional particular case $(n,k)=(6,3)$ was also fully calculated (it suffices to study $\widetilde G_{n,k}$ with $n\ge2k$, since $\widetilde G_{n,k}\cong\widetilde G_{n,n-k}$).

For $k=3$ and $n\in\{2^t-3,2^t-2,2^t-1,2^t\}$ ($t\ge3$), significant results in this regard were obtained by Basu and Chakraborty in \cite{BasuChakraborty}. In that paper the authors showed that, besides Stiefel--Whitney classes $\widetilde w_2$ and $\widetilde w_3$, there is exactly one (up to addition of a polynomial in $\widetilde w_2$ and $\widetilde w_3$) indecomposable class $\widetilde a$ in $H^*(\widetilde G_{n,3};\mathbb Z_2)$, and they identified the algebra $H^*(\widetilde G_{n,3};\mathbb Z_2)$ as the quotient of the polynomial algebra $\mathbb Z_2[w_2,w_3,a]$ by a certain ideal. In the generating set of that ideal some polynomials in $w_2$ and $w_3$ remained undetermined. This description was completed in the case $n=2^t$ ($t\ge3$) in \cite{CP}. For $n\in\{2^t-3,2^t-2,2^t-1\}$ ($t\ge3$), the description is essentially the following (\cite[Theorem B(2)]{BasuChakraborty}):
\begin{equation}\label{isomorphism}
H^*(\widetilde G_{n,3};\mathbb Z_2)\cong\frac{\mathbb Z_2[w_2,w_3,a_{2^t-4}]}{\big(g_{n-2},g_{n-1},g_n,a_{2^t-4}^2-P_1a_{2^t-4}-P_2\big)},
\end{equation}
where $P_1$ and $P_2$ are some (unknown) polynomials in $w_2$ and $w_3$, with $P_2=0$ for $n\in\{2^t-3,2^t-2\}$. Via the isomorphism (\ref{isomorphism}), the cosets of $w_2$ and $w_3$ from the right-hand side correspond to the Stiefel--Whitney classes $\widetilde w_2\in H^2(\widetilde G_{n,3};\mathbb Z_2)$ and $\widetilde w_3\in H^3(\widetilde G_{n,3};\mathbb Z_2)$ respectively, while the coset of $a_{2^t-4}$ corresponds to the indecomposable class $\widetilde a_{2^t-4}\in H^{2^t-4}(\widetilde G_{n,3};\mathbb Z_2)$. The polynomials $g_r$, $r\ge0$, appearing in (\ref{isomorphism}), are well-known polynomials in $w_2$ and $w_3$.

This paper deals with the problem of computing the polynomials $P_1$ and $P_2$ from (\ref{isomorphism}). For $n\in\{2^t-3,2^t-2\}$, there is actually only $P_1$ to be determined, and we prove that $P_1=0$ in these cases (more precisely, that $P_1$ can be chosen to be the zero-polynomial in order for the isomorphism (\ref{isomorphism}) to hold). This means that $H^*(\widetilde G_{n,3};\mathbb Z_2)$ is isomorphic to the tensor product of $\mathbb Z_2[w_2,w_3]/(g_{n-2},g_{n-1},g_n)$ and the exterior algebra $\Lambda(a_{2^t-4})\cong\mathbb Z_2[a_{2^t-4}]/(a_{2^t-4}^2)$, that is, $\widetilde a_{2^t-4}^2=0$ in $H^*(\widetilde G_{n,3};\mathbb Z_2)$. Therefore, we obtain a complete description of $H^*(\widetilde G_{n,3};\mathbb Z_2)$ in these cases. In the case $n=2^t-1$ we prove that $\widetilde a_{2^t-4}^2\neq0$ in $H^*(\widetilde G_{n,3};\mathbb Z_2)$, and establish that $P_1=g_{2^t-4}$ and $P_2=\gamma w_2^{2^t-4}$ for some $\gamma\in\mathbb Z_2$. So the main results of this paper are given in the following theorem (it is well known that $g_{2^t-3}=0$, and so this polynomial is redundant in (\ref{isomorphism})).

\begin{theorem}\label{mainthm}
For all integers $t\ge3$ we have the following isomorphisms of graded algebras (where $|w_i|=i$, $i=2,3$, and $|a_{2^t-4}|=2^t-4$):
\begin{itemize}
\item[(a)] $\displaystyle H^*(\widetilde G_{2^t-1,3};\mathbb Z_2)\cong\frac{\mathbb Z_2[w_2,w_3,a_{2^t-4}]}{\big(g_{2^t-2},g_{2^t-1},a_{2^t-4}^2-g_{2^t-4}a_{2^t-4}-\gamma w_2^{2^t-4}\big)}$,

    \noindent for some $\gamma\in\mathbb Z_2$;
\item[(b)] $\displaystyle H^*(\widetilde G_{2^t-2,3};\mathbb Z_2)\cong\frac{\mathbb Z_2[w_2,w_3]}{(g_{2^t-4},g_{2^t-2})}\otimes\Lambda(a_{2^t-4})$;
\item[(c)] $\displaystyle H^*(\widetilde G_{2^t-3,3};\mathbb Z_2)\cong\frac{\mathbb Z_2[w_2,w_3]}{(g_{2^t-5},g_{2^t-4})}\otimes\Lambda(a_{2^t-4})$.
\end{itemize}
\end{theorem}

The technique that we use to prove this theorem combines the theory of Gr\"obner bases with some classical methods of algebraic topology. In Section \ref{preliminaries} of the paper we establish a framework for the proof of Theorem \ref{mainthm} by recalling some basic facts concerning the cohomology algebras $H^*(G_{n,k};\mathbb Z_2)$ and $H^*(\widetilde G_{n,k};\mathbb Z_2)$, in particular the Gysin sequence connecting these two cohomologies. In Section \ref{2t-1} we use the Gr\"obner basis for the ideal $(g_{2^t-2},g_{2^t-1})\trianglelefteq\mathbb Z_2[w_2,w_3]$ obtained by Fukaya in \cite{Fukaya}, as well as an additive basis for $H^*(\widetilde G_{2^t-1,3};\mathbb Z_2)$ that it produces, to prove part (a) of the theorem. Finally, Sections \ref{2t-2} and \ref{2t-3} prove parts (b) and (c) of the theorem.

\section{Some preliminaries on cohomology of Grassmann manifolds}
\label{preliminaries}

We will be working with $\mathbb Z_2$ coefficients exclusively, so in the rest of the paper, these coefficients for cohomology are understood.

By Borel's description \cite{Borel}, as a graded algebra
\begin{equation}\label{kohomologija_grasmanijana}
H^*(G_{n,k})\cong\mathbb Z_2[w_1,w_2,\ldots,w_k]/I_{n,k},
\end{equation}
where the coset of $w_i$ in the quotient on the right-hand side corresponds to the Stiefel--Whitney class $w_i(\gamma_{n,k})\in H^i(G_{n,k})$ on the left-hand side. So the grading on $\mathbb Z_2[w_1,w_2,\ldots,w_k]$ is such that the degree of $w_i$ is $i$, $1\le i\le k$. $I_{n,k}$ is the graded ideal in $\mathbb Z_2[w_1,w_2,\ldots,w_k]$ generated by polynomials $\overline w_{n-k+1},\ldots,\overline w_n$, obtained as appropriate homogeneous parts of the power series $1/(1+w_1+w_2+\cdots+w_k)$. Therefore, they satisfy the relation
\begin{equation}\label{power_series}
(1+w_1+w_2+\cdots+w_k)(\overline w_0+\overline w_1+\overline w_2+\cdots)=1.    
\end{equation}

Let $p:\widetilde G_{n,k}\rightarrow G_{n,k}$ be the map that forgets the orientation of the $k$-plane in $\mathbb R^n$. It is well known that $p$ is a two-fold (universal) covering map, and that for the induced map $p^*:H^*(G_{n,k})\rightarrow H^*(\widetilde G_{n,k})$ one has $p^*(w_i(\gamma_{n,k}))=w_i(\widetilde\gamma_{n,k})=\widetilde w_i$, $1\le i\le k$. Therefore, $\im p^*$ is a subalgebra of $H^*(\widetilde G_{n,k})$ generated by $\widetilde w_2,\ldots,\widetilde w_k$ ($\widetilde w_1=0$ since $\widetilde G_{n,k}$ is simply connected). The Gysin exact sequence associated to $p$ has the following form:
\[\begin{tikzcd}[column sep = 1em, font = \small, cramped]
\cdots
\arrow{r}
&
H^r(G_{n,k})
\arrow{r}{p^*}
&
H^r(\widetilde G_{n,k})
\arrow{r}
&
H^r(G_{n,k})
\arrow{r}{w_1}
&
H^{r+1}(G_{n,k})
\arrow{r}{p^*}
&
H^{r+1}(\widetilde G_{n,k})
\arrow{r}
&
\cdots,
\end{tikzcd}\]
where $H^r(G_{n,k})\stackrel{w_1}\longrightarrow H^{r+1}(G_{n,k})$ is the multiplication with $w_1(\gamma_{n,k})$.

The inclusion $\mathbb R^n\hookrightarrow\mathbb R^{n+1}$ induces embeddings $i:G_{n,k}\hookrightarrow G_{n+1,k}$ and $j:G_{n,k}\hookrightarrow G_{n+1,k+1}$; and likewise in the oriented case, it induces $\widetilde i:\widetilde G_{n,k}\hookrightarrow\widetilde G_{n+1,k}$ and $\widetilde j:\widetilde G_{n,k}\hookrightarrow\widetilde G_{n+1,k+1}$. Moreover, these embeddings commute with the double covering $p$, so we have the following commutative diagrams of two-fold coverings ($S^0$-bundles).
\[
\begin{tikzcd}
S^0
\arrow[equal]{r}
\arrow[hookrightarrow]{d}
&
S^0
\arrow[hookrightarrow]{d}
&
S^0
\arrow[equal]{r}
\arrow[hookrightarrow]{d}
&
S^0
\arrow[hookrightarrow]{d}
\\
\widetilde G_{n,k}
\arrow[hookrightarrow]{r}{\widetilde i}
\arrow{d}[left]{p}
&
\widetilde G_{n+1,k}
\arrow{d}{p}
&
\widetilde G_{n,k}
\arrow[hookrightarrow]{r}{\widetilde j}
\arrow{d}[left]{p}
&
\widetilde G_{n+1,k+1}
\arrow{d}{p}
\\
G_{n,k}
\arrow[hookrightarrow]{r}{i}
&
G_{n+1,k}
&
G_{n,k}
\arrow[hookrightarrow]{r}{j}
&
G_{n+1,k+1}
\end{tikzcd}
\]
These diagrams, in return, induce morphisms between Gysin sequences.
\begin{equation}\label{gysin1}
\begin{tikzcd}[column sep = 1em, font = \small, cramped]
\cdots
\arrow{r}
&
H^r(G_{n+1,k})
\arrow{r}{p^*}
\arrow{d}{i^*}
&
H^r(\widetilde G_{n+1,k})
\arrow{r}{\varphi}
\arrow{d}{\widetilde i^*}
&
H^r(G_{n+1,k})
\arrow{r}{w_1}
\arrow{d}{i^*}
&
H^{r+1}(G_{n+1,k})
\arrow{r}{}
\arrow{d}{i^*}
&
\cdots
\\
\cdots
\arrow{r}
&
H^r(G_{n,k})
\arrow{r}{p^*}
&
H^r(\widetilde G_{n,k})
\arrow{r}{\varphi}
&
H^r(G_{n,k})
\arrow{r}{w_1}
&
H^{r+1}(G_{n,k})
\arrow{r}
&
\cdots
\end{tikzcd}
\end{equation}
\begin{equation}\label{gysin2}
\begin{tikzcd}[column sep = 1em, font = \small, cramped]
\cdots
H^r(G_{n+1,k+1})
\arrow{r}{p^*}
\arrow{d}{j^*}
&
H^r(\widetilde G_{n+1,k+1})
\arrow{r}{\varphi}
\arrow{d}{\widetilde j^*}
&
H^r(G_{n+1,k+1})
\arrow{r}{w_1}
\arrow{d}{j^*}
&
H^{r+1}(G_{n+1,k+1})
\arrow{r}{}
\arrow{d}{j^*}
&
\cdots
\\
\cdots
H^r(G_{n,k})
\arrow{r}{p^*}
&
H^r(\widetilde G_{n,k})
\arrow{r}{\varphi}
&
H^r(G_{n,k})
\arrow{r}{w_1}
&
H^{r+1}(G_{n,k})
\arrow{r}
&
\cdots
\end{tikzcd}
\end{equation}

The following lemma will be used repeatedly in the upcoming sections.

\begin{lemma}\label{osnovna_lema}
{\rm(a)} Suppose that $r\ge0$ is an integer with the property that in $H^r(G_{n+1,k})$ one has $\ker w_1\cap\ker i^*=0$ (see (\ref{gysin1})). Then for an arbitrary class $x\in H^r(\widetilde G_{n+1,k})$ the following implication holds:
 \[x\notin\im p^* \quad \Longrightarrow \quad \widetilde i^*(x)\notin\im p^*.\]

 {\rm(b)} Similarly, if $r\ge0$ is an integer with the property that $\ker w_1\cap\ker j^*=0$ in $H^r(G_{n+1,k+1})$ (see (\ref{gysin2})), then for a class $x\in H^r(\widetilde G_{n+1,k+1})$ the following implication holds:
 \[x\notin\im p^* \quad \Longrightarrow \quad \widetilde j^*(x)\notin\im p^*.\]
\end{lemma}
\begin{proof}
    The proof of the lemma comes down to elementary diagram chasing. For (a) we use diagram (\ref{gysin1}). If $x\notin\im p^*=\ker\varphi$, then $\varphi(x)\neq0$ and $\varphi(x)\in\ker w_1$. From $\ker w_1\cap\ker i^*=0$ we now conclude that $\varphi(x)\notin\ker i^*$, and so
    \[\varphi\big(\,\,\widetilde i^*(x)\big)=i^*\big(\varphi(x)\big)\neq0,\]
    which means that $\widetilde i^*(x)\notin\ker\varphi=\im p^*$.

    The proof of part (b) is identical (one just uses diagram (\ref{gysin2})).
\end{proof}

Hence, we will be interested in the kernels of the maps $i^*:H^*(G_{n+1,k})\rightarrow H^*(G_{n,k})$ and $j^*:H^*(G_{n+1,k+1})\rightarrow H^*(G_{n,k})$. The total space of canonical vector bundle $\gamma_{n,k}$ over $G_{n,k}$ consists of pairs $(V,v)$, where $V\in G_{n,k}$ and $v\in V$, so it is easy to see that the embeddings $i:G_{n,k}\hookrightarrow G_{n+1,k}$ and $j:G_{n,k}\hookrightarrow G_{n+1,k+1}$ are covered by bundle maps from $\gamma_{n,k}$ to $\gamma_{n+1,k}$ and from $\gamma_{n,k}\oplus\varepsilon^1$ to $\gamma_{n+1,k+1}$ respectively ($\varepsilon^1$ is the trivial line bundle over $G_{n,k}$). Therefore, the next lemma is straightforward from the naturality of the Stiefel--Whitney classes, the equality $w_r(\gamma_{n,k}\oplus\varepsilon^1)=w_r(\gamma_{n,k})$, and the fact that the vector bundle $\gamma_{n,k}$ is $k$-dimensional (see \cite[p.\ 37ff]{MilnorSt}).

\begin{lemma}\label{naturality}
{\rm(a)} For the map $i^*:H^*(G_{n+1,k})\rightarrow H^*(G_{n,k})$ we have
 \[i^*\big(w_r(\gamma_{n+1,k})\big)=w_r(\gamma_{n,k}), \quad 1\le r\le k.\]

 {\rm(b)} For the map $j^*:H^*(G_{n+1,k+1})\rightarrow H^*(G_{n,k})$ we have
 \begin{align*}
  j^*\big(w_r(\gamma_{n+1,k+1})\big)&=w_r(\gamma_{n,k}), \quad 1\le r\le k,\\
  j^*\big(w_{k+1}(\gamma_{n+1,k+1})\big)&=w_{k+1}(\gamma_{n,k})=0.   
 \end{align*}
\end{lemma}

The same arguments work in the oriented case as well, and so for embeddings $\widetilde i:\widetilde G_{n,k}\hookrightarrow\widetilde G_{n+1,k}$ and $\widetilde j:\widetilde G_{n,k}\hookrightarrow\widetilde G_{n+1,k+1}$ one has
\begin{equation}\label{naturalityi}
 \widetilde i^*(\widetilde w_r)=\widetilde w_r, \quad 2\le r\le k,  
\end{equation}
and
\begin{equation}\label{naturalityj}
 \widetilde j^*(\widetilde w_r)=\widetilde w_r, \quad 2\le r\le k; \qquad  \widetilde j^*(\widetilde w_{k+1})=0.
\end{equation}

\medskip

In the case $k=3$, equality (\ref{power_series}) leads to the recurrence relation
\begin{equation}\label{w_rekurzivno}
    \overline w_{r+3} = w_1\overline w_{r+2} + w_2\overline w_{r+1} + w_3\overline w_r,\quad r\ge 0,
\end{equation}
and also an explicit formula for calculating $\overline w_r$. Specifically,
\begin{equation}\label{overline_w}
    \overline w_r = \sum_{a+2b+3c = r} \binom{a+b+c}{a}\binom{b+c}{b} w_1^aw_2^bw_3^c,\quad r\ge 0,
\end{equation}
where the sum is taken over all triples $(a,b,c)$ of nonnegative integers such that $a+2b+3c = r$ (both binomial coefficients in the sum are considered modulo $2$).

It is well known (see \cite{Fukaya} or \cite{CP}) that in the case $k=3$
\begin{equation}\label{isomorphism_imp*}
\im p^*\cong\mathbb Z_2[w_2,w_3]/J_{n,3},
\end{equation}
where $J_{n,3}$ is the graded ideal in $\mathbb Z_2[w_2,w_3]$ generated by the polynomials $g_{n-2}$, $g_{n-1}$ and $g_n$. These polynomials $g_r\in\mathbb Z_2[w_2,w_3]$, $r\ge0$, are modulo $w_1$ reductions of the polynomials $\overline w_r\in\mathbb Z_2[w_1,w_2,w_3]$. From (\ref{overline_w}) we now get
\begin{equation}\label{gpolk3}
g_r=\sum_{2b+3c=r}{b+c\choose b}\,w_2^bw_3^c, \quad r\ge0,
\end{equation}
and (\ref{w_rekurzivno}) becomes
\begin{equation}\label{recgpolk3}
g_{r+3}=w_2g_{r+1}+w_3g_r, \quad r\ge0.
\end{equation}
By definition, $J_{n,3}=(g_{n-2},g_{n-1},g_n)$, but from (\ref{recgpolk3}) we see that actually
\begin{equation}\label{r>n-2}
g_r\in J_{n,3} \quad \mbox{ for all } r\ge n-2.
\end{equation}

For a (homogeneous) polynomial $f\in\mathbb Z_2[w_2,w_3]$, we shall denote by $\widetilde f\in\im p^*\subset H^*(\widetilde G_{n,3})$ the cohomology class corresponding to (the coset of) $f$ via the isomorphism (\ref{isomorphism_imp*}). This is in accordance with the notation $\widetilde w_i$, that we use for the Stiefel--Whitney class $w_i(\widetilde\gamma_{n,3})$, $i=2,3$. For instance, the relation (\ref{r>n-2}) can now be rewritten as
\begin{equation}\label{r>=n-2}
\widetilde g_r=0 \quad \mbox{ in } H^*(\widetilde G_{n,3}),  \quad \mbox{ for all } r\ge n-2.
\end{equation}

\section{The case $n=2^t-1$}
\label{2t-1}

Let $t\ge3$ be a fixed integer. By \cite[Theorem A(b)]{BasuChakraborty}, in $H^*(\widetilde G_{2^t-1,3})$, besides $\widetilde w_2$ and $\widetilde w_3$, there is one more indecomposable class $\widetilde a_{2^t-4}$ in cohomological dimension $2^t-4$. \em Indecomposable \em means that it cannot be expressed as a polynomial in classes of (strictly) smaller dimension. Since all classes of smaller dimension are polynomials in $\widetilde w_2$ and $\widetilde w_3$, and the subalgebra $\im p^*$ is generated by these two Stiefel--Whitney classes, this is equivalent to the statement $\widetilde a_{2^t-4}\notin\im p^*$. The uniqueness of $\widetilde a_{2^t-4}$ is only up to addition of an element of $\im p^*$. Put in other words, if $\widetilde a_{2^t-4}$ is an indecomposable class in $H^{2^t-4}(\widetilde G_{2^t-1,3})$, then so is $\widetilde a_{2^t-4}+\widetilde w$ for any $\widetilde w\in H^{2^t-4}(\widetilde G_{2^t-1,3})\cap\im p^*$. This means that for $\widetilde a_{2^t-4}$ we can choose any class from $H^{2^t-4}(\widetilde G_{2^t-1,3})\setminus\im p^*$.

We are going to make this choice in the following way. By \cite[Theorem A(a)]{BasuChakraborty} we have an indecomposable class (outside $\im p^*$) $\widetilde a_{2^t-1}\in H^{2^t-1}(\widetilde G_{2^t,3})$. We will prove that $\widetilde i^*(\widetilde a_{2^t-1})\in H^{2^t-1}(\widetilde G_{2^t-1,3})\setminus\im p^*$, where $\widetilde i:\widetilde G_{2^t-1,3}\hookrightarrow\widetilde G_{2^t,3}$ is the embedding from the previous section. Since every class in $H^*(\widetilde G_{2^t-1,3})$ is a polynomial in $\widetilde w_2$, $\widetilde w_3$ and some (indecomposable) class in $H^{2^t-4}(\widetilde G_{2^t-1,3})$ (by \cite[Theorem A(b)]{BasuChakraborty}), we see that an odd-dimensional class must be divisible by $\widetilde w_3$. Hence, $\widetilde i^*(\widetilde a_{2^t-1})$ must be of the form $\widetilde w_3\sigma$, for some $\sigma\in H^{2^t-4}(\widetilde G_{2^t-1,3})$, and if we show $\widetilde i^*(\widetilde a_{2^t-1})\notin\im p^*$, we will have $\sigma\notin\im p^*$. Then we will be able to take $\widetilde a_{2^t-4}:=\sigma$ and thus ensure the property:
\begin{equation}\label{i*2^t-->2^t-1}
\widetilde i^*(\widetilde a_{2^t-1})=\widetilde w_3\widetilde a_{2^t-4}.
\end{equation}

Therefore, we are left to prove that $\widetilde i^*(\widetilde a_{2^t-1})\notin\im p^*$. This will be an immediate consequence of Lemma \ref{osnovna_lema}(a) (for $r=n=2^t-1$ and $k=3$) as soon as we establish the following proposition.

\begin{proposition}\label{kercapker2} Let $H^{2^t-1}(G_{2^t,3})\stackrel{w_1}\longrightarrow H^{2^t}(G_{2^t,3})$ be the multiplication with the Stiefel--Whitney class $w_1(\gamma_{2^t,3})$ and $i^*:H^{2^t-1}(G_{2^t,3})\rightarrow H^{2^t-1}(G_{2^t-1,3})$ the map induced by the embedding $i:G_{2^t-1,3}\hookrightarrow G_{2^t,3}$. Then
    $\ker w_1\cap \ker i^* = 0$.
\end{proposition}
\begin{proof}
    For a homogeneous polynomial $f\in\mathbb Z_2[w_1,w_2,w_3]$ we will denote by $[f]$ the cohomology class corresponding to $f$ via the isomorphism (\ref{kohomologija_grasmanijana}). If $[f]\in\ker w_1\cap \ker i^*$, then by Lemma \ref{naturality}(a) we first have $[f] = i^*[f] = 0$ in $H^{2^t-1}(G_{2^t-1,3})$, so 
    \[f = (\alpha w_1^2 + \beta w_2)\overline w_{2^t-3} + \gamma w_1\overline w_{2^t-2} + \delta \overline w_{2^t-1},\]
    for some $\alpha,\beta,\gamma,\delta\in\mathbb Z_2$. Since $\overline w_{2^t-2}, \overline w_{2^t-1}\in J_{2^t,3}$,
    \[[f] = [(\alpha w_1^2 + \beta w_2)\overline w_{2^t-3} + \gamma w_1\overline w_{2^t-2} + \delta \overline w_{2^t-1}] = [(\alpha w_1^2 + \beta w_2)\overline w_{2^t-3}],\]
    so we can choose $\gamma = \delta = 0$. 

    On the other hand, since $[f]\in \ker w_1$, we have $[w_1f] = 0$ in $H^{2^t}(G_{2^t,3})$, so
    \begin{equation}\label{w1f}
        w_1f = (\lambda w_1^2 + \mu w_2) \overline w_{2^t-2} + \nu w_1 \overline w_{2^t-1} + \varepsilon \overline w_{2^t},
    \end{equation}
    for some $\lambda, \mu, \nu, \varepsilon\in\mathbb Z_2$. By (\ref{w_rekurzivno}),
    \[\overline w_{2^t} = w_1\overline w_{2^t-1} + w_2\overline w_{2^t-2} + w_3\overline w_{2^t-3},\]
    and after substituting $(\alpha w_1^2 + \beta w_2)\overline w_{2^t-3}$ for $f$, (\ref{w1f}) becomes
    \begin{equation} \label{w1f1}
        (\alpha w_1^3 + \beta w_1 w_2 + \varepsilon w_3)\overline w_{2^t-3} = (\lambda w_1^2 + (\mu+ \varepsilon) w_2) \overline w_{2^t-2} + (\nu + \varepsilon)w_1\overline w_{2^{t-1}}.
    \end{equation}
By (\ref{overline_w}), the polynomial $\overline w_{2^t-2}$ contains the monomial $w_2^{2^{t-1}-1}$ with nonzero coefficient, and it is obvious that $(\mu+ \varepsilon) w_2\overline w_{2^t-2}$ is the only summand in (\ref{w1f1}) that can contain $w_2^{2^{t-1}}$. So we conclude $\mu+ \varepsilon=0$, i.e., $\mu = \varepsilon$. Equation (\ref{w1f1}) now becomes
    \begin{equation}\label{w1f2}
        (\alpha w_1^3 + \beta w_1 w_2 + \mu w_3)\overline w_{2^t-3} = \lambda w_1^2 \overline w_{2^t-2} + (\nu + \mu)w_1\overline w_{2^{t-1}}.
    \end{equation}
We now use the same idea to show that $\alpha =\beta = 0$. For $\alpha$ we observe the monomial $w_1^4w_2^{2^{t-1}-2}$, and use (\ref{overline_w}) in order to determine the coefficients of this monomial in each of the summands in (\ref{w1f2}). These coefficients are listed in Table \ref{tab:alpha}.
\begin{table}[h]
\renewcommand{\arraystretch}{1.5} %
\begin{tabular}{|c|c|c|c|c|}
 \hline
summand & $a$ & $b$ & $c$ &  $\binom{a+b+c}{a}\binom{b+c}{b}$ \\ [1ex]
 \hline
$\alpha w_1^3\overline w_{2^t-3}$ & 1 & $2^{t-1}-2$ & 0 &   1\\ 
$\beta w_1w_2\overline w_{2^t-3}$ & 3 & $2^{t-1}-3$ & 0 &   0\\
$\mu w_3\overline w_{2^t-3}$ & - & - & - &   -\\
$\lambda w_1^2 \overline w_{2^t-2}$ & 2 & $2^{t-1}-2$ & 0 &   0\\
$(\nu + \mu)w_1\overline w_{2^{t-1}}$ & 3 & $2^{t-1}-2$ & 0 &   0\\
 \hline
  \end{tabular}
\caption{\label{tab:alpha} Coefficients of $w_1^4w_2^{2^{t-1}-2}$ in (\ref{w1f2})}
\end{table}
So $\alpha w_1^3\overline w_{2^t-3}$ is the only summand containing $w_1^4w_2^{2^{t-1}-2}$, and its coefficient is $\alpha$. We conclude that $\alpha =0$.

For $\beta$ we proceed in the same way. We now observe the monomial $w_1^3w_2^{2^{t-1}-3}w_3$, and list its coefficients in Table \ref{tab:beta}, which allows us to conclude that $\beta =0$.
\begin{table}[h]
\renewcommand{\arraystretch}{1.5} %
\begin{tabular}{|c|c|c|c|c|}
 \hline
summand & $a$ & $b$ & $c$ &  $\binom{a+b+c}{a}\binom{b+c}{b}$ \\ [1ex]
 \hline
$\alpha w_1^3\overline w_{2^t-3}$ & 0 & $2^{t-1}-3$ & 1 &   0\\ 
$\beta w_1w_2\overline w_{2^t-3}$ & 2 & $2^{t-1}-4$ & 1 &   1\\
$\mu w_3\overline w_{2^t-3}$ & 3 & $2^{t-1}-3$ & 0 &   0\\
$\lambda w_1^2 \overline w_{2^t-2}$ & 1 & $2^{t-1}-3$ & 1 &   0\\
$(\nu + \mu)w_1\overline w_{2^{t-1}}$ & 2 & $2^{t-1}-3$ & 1 &   0\\
 \hline
  \end{tabular}
\caption{\label{tab:beta} Coefficients of $w_1^3w_2^{2^{t-1}-3}w_3$ in (\ref{w1f2})}
\end{table}

 Finally, we have $[f] = [(\alpha w_1^2 + \beta w_2)\overline w_{2^t-3}] = 0$ in $H^{2^t-1}(G_{2^t,3})$, and since $[f]$ was an arbitrary element of $\ker w_1\cap \ker i^*$, this concludes the proof.
\end{proof}

We have thus chosen indecomposable classes $\widetilde a_{2^t-1}\in H^{2^t-1}(\widetilde G_{2^t,3})$ and $\widetilde a_{2^t-4}\in H^{2^t-4}(\widetilde G_{2^t-1,3})$ such that (\ref{i*2^t-->2^t-1}) holds.

It is well known (see \cite[Lemma 2.1(i)]{BasuChakraborty}) that $g_{2^t-3}=0$, so we get $J_{2^t-1,3}=(g_{2^t-3},g_{2^t-2},g_{2^t-1})=(g_{2^t-2},g_{2^t-1})$. Also, according to \cite[Theorem B(2)]{BasuChakraborty} we have an isomorphism of graded algebras
\begin{equation}\label{isomorphism2^t-1}
 H^*(\widetilde G_{2^t-1,3})\cong\frac{\mathbb Z_2[w_2,w_3,a_{2^t-4}]}{\big(g_{2^t-2},g_{2^t-1},a_{2^t-4}^2-P_1a_{2^t-4}-P_2\big)},  
 \end{equation}
 mapping the classes $\widetilde w_2$, $\widetilde w_3$ and $\widetilde a_{2^t-4}$ to the cosets of $w_2$, $w_3$ and $a_{2^t-4}$ respectively, where the (unknown) polynomials $P_1$ and $P_2$ do not contain the variable $a_{2^t-4}$. 

In \cite{Fukaya} Fukaya found a Gr\"obner basis for the ideal $J_{2^t-1,3}$, and in \cite{CP} it was shown that $J_{2^t-1,3}=J_{2^t,3}$ and that this Gr\"obner basis consists of polynomials
\begin{equation}\label{grebnerova_baza}
    f_i=g_{2^t-3+2^i}=\sum_{2b+3c=2^t-3+2^i}{b+c\choose b}w_2^bw_3^c,\quad \mbox{ for } 0\le i\le t-1.
    \end{equation}
These polynomials, of course, belong to $J_{2^t-1,3}$ (see (\ref{r>n-2})), and so we have that $\widetilde f_i=0$ in $H^*(\widetilde G_{2^t-1,3})$ for all $i$. By \cite[p.\ 203]{Fukaya} (or \cite[Proposition 3.4]{CP}) we know that the leading monomial (with respect to the lexicographic order with $w_2>w_3$) of $f_i$ is the following:
\begin{equation}\label{LMfi}
\mathrm{LM}(f_i)=w_2^{2^{t-1}-2^i}w_3^{2^i-1}.
\end{equation}
Moreover, \cite[Proposition 3.4]{CP} gives us that
\begin{equation}\label{ft-1}
f_{t-1}=\mathrm{LM}(f_{t-1})=w_3^{2^{t-1}-1}.
\end{equation}

In the same way as in \cite{CP} we can use this Gr\"obner basis and the isomorphism (\ref{isomorphism2^t-1}) to find an additive basis for the cohomology algebra $H^*(\widetilde G_{2^t-1,3})$. First, if we choose the lexicographic order with $a_{2^t-4}>w_2>w_3$ for the monomial order in $\mathbb Z_2[w_2,w_3,a_{2^t-4}]$, then the set $\{f_0,f_1,\ldots,f_{t-1},a_{2^t-4}^2-P_1a_{2^t-4}-P_2\}$ turns out to be a Gr\"obner basis for the ideal $(g_{2^t-2},g_{2^t-1},a_{2^t-4}^2-P_1a_{2^t-4}-P_2)$ in $\mathbb Z_2[w_2,w_3,a_{2^t-4}]$ (whatever the polynomials $P_1$ and $P_2$ might be). This is essentially due to the fact that the leading monomial of the polynomial $a_{2^t-4}^2-P_1a_{2^t-4}-P_2$ is $a_{2^t-4}^2$ (since $P_1$ and $P_2$ are polynomials in $w_2$ and $w_3$ only), and $a_{2^t-4}^2$ is coprime with every $\mathrm{LM}(f_i)$, $0\le i\le t-1$ (cf.\ \cite[Theorem 3.6]{CP}). On the other hand, if we observe the monomials not divisible by any of the leading monomials from Gr\"obner basis, it is well known that their cosets form an additive basis for the quotient. Therefore, the isomorphism (\ref{isomorphism2^t-1}) leads to the conclusion that the set
\begin{equation}\label{aditivna_baza}
    B = \left\{\widetilde a_{2^t-4}^r \widetilde w_2^b \widetilde w_3^c \mid r<2,\, \left(\forall i\in\{0,1,\ldots,t-1\}\right)\, b<2^{t-1}-2^i \vee c<2^i-1\right\}\end{equation}
is an additive basis for $H^*(\widetilde G_{2^t-1,3})$ (cf.\ \cite[Corollary 3.6.1]{CP}).

The elements of $B$ with $r=0$, i.e., the classes $\widetilde w_2^b\widetilde w_3^c$ such that for all $i\in\{0,1,\ldots,t-1\}$ either $b<2^{t-1}-2^i$ or $c<2^i-1$, form an additive basis for $\im p^*$ (this could be obtained from the isomorphism (\ref{isomorphism_imp*}) by an analogous consideration). Now it is not hard to see that the basis element from $\im p^*$ of the highest cohomological dimension is $\widetilde w_2^{2^{t-2}-1}\widetilde w_3^{2^{t-1}-2}\in H^{2^{t+1}-8}(\widetilde G_{2^t-1,3})$. Indeed, for $\widetilde w_2^b\widetilde w_3^c\in B$ one must have $c\le2^{t-1}-2$ (since $c\ge2^{t-1}-1$ would imply $b<0$), and if $i\in\{0,1,\ldots,t-2\}$ is such that $2^i-1\le c\le2^{i+1}-2$, then $b\le2^{t-1}-2^i-1$. So
\begin{equation}\label{ineq}
2b+3c\le2(2^{t-1}-2^i-1)+3(2^{i+1}-2)=2^t+2^{i+2}-8\le2^{t+1}-8.
\end{equation}
This means that the following implication holds
\begin{equation}\label{propimp*}
 j>2^{t+1}-8 \quad \Longrightarrow \quad H^j(\widetilde G_{2^t-1,3})\cap\im p^*=0.  
\end{equation}

Note also that $\widetilde w_2^{2^{t-2}-1}\widetilde w_3^{2^{t-1}-2}$ is the \em only \em element in $B\cap\im p^*$ of the cohomological dimension $2^{t+1}-8$. Namely, in order to have equality in (\ref{ineq}) one must have $b=2^{t-1}-2^i-1$, $c=2^{i+1}-2$ and $i=t-2$. On the other hand, by \cite[Corollary 4.12]{Fukaya} we know that $\widetilde w_2^{2^t-4}\neq0$, and so we conclude
\begin{equation}\label{visina_w_2}
 \widetilde w_2^{2^t-4}=\widetilde w_2^{2^{t-2}-1}\widetilde w_3^{2^{t-1}-2}\neq0 \quad \mbox{ in } H^{2^{t+1}-8}(\widetilde G_{2^t-1,3}).  
\end{equation}

Similarly, the only element of $B$ in the top dimension $3\cdot2^t-12$ (this is the dimension of the manifold $\widetilde G_{2^t-1,3}$) is $\widetilde a_{2^t-4}\widetilde w_2^{2^{t-2}-1}\widetilde w_3^{2^{t-1}-2}$, so we have
\begin{equation}\label{topdim}
 \widetilde a_{2^t-4}\widetilde w_2^{2^t-4}=\widetilde a_{2^t-4}\widetilde w_2^{2^{t-2}-1}\widetilde w_3^{2^{t-1}-2}\neq0 \quad \mbox{ in } H^{3\cdot2^t-12}(\widetilde G_{2^t-1,3})\cong\mathbb Z_2.  
\end{equation}

\medskip

Theorem \ref{mainthm}(a) states that in (\ref{isomorphism2^t-1}) we can take $P_1=g_{2^t-4}$ and $P_2=\gamma w_2^{2^t-4}$ for some $\gamma\in\mathbb Z_2$. In order to prove this it suffices to show that in $H^*(\widetilde G_{2^t-1,3})$ we have
\begin{equation}\label{eq}
\widetilde a_{2^t-4}^2=\widetilde g_{2^t-4}\widetilde a_{2^t-4}+\gamma\widetilde w_2^{2^t-4}
\end{equation}
(for some $\gamma\in\mathbb Z_2$).

We start off by representing $\widetilde a_{2^t-4}^2$ in the basis $B$ (see (\ref{aditivna_baza})). Since $\widetilde a_{2^t-4}^2\in H^{2^{t+1}-8}(\widetilde G_{2^t-1,3})$, we need to list basis elements in this cohomological dimension. As we have already noticed, from $\im p^*$ we have $\widetilde w_2^{2^{t-2}-1}\widetilde w_3^{2^{t-1}-2}$, and outside of $\im p^*$ we have $\widetilde a_{2^t-4} \widetilde w_2^{2^{t-1}-2-3k}\widetilde w_3^{2k}$ for every $k\ge0$ such that $2^{t-1}-2-3k\ge0$, i.e., $k\le(2^{t-1}-2)/3$. Indeed, $\widetilde a_{2^t-4} \widetilde w_2^{2^{t-1}-2-3k}\widetilde w_3^{2k}\in B$ for all such $k$, since $2^{t-1}-2-3k\ge2^{t-1}-2^i$ and $2k\ge2^i-1$ for some $i\in\{0,1,\ldots,t-1\}$ would imply $2^i\ge3k+2$ and $2^i\le2k+1$, which is obviously impossible. Having in mind (\ref{visina_w_2}), we thus have
\begin{equation}\label{a^2_preko_baznih}
    \widetilde a_{2^t-4}^2 = \gamma\widetilde w_2^{2^t-4} + \sum_{k=0}^{\left\lfloor \frac{2^{t-1}-2}{3} \right\rfloor} \lambda_k \widetilde a_{2^t-4} \widetilde w_2^{2^{t-1}-2-3k}\widetilde w_3^{2k},
\end{equation}
for some uniquely determined $\gamma,\lambda_0,\lambda_1,\ldots\in\mathbb Z_2$. In order to obtain (\ref{eq}), we see that we need to equate the sum on the right-hand side with $\widetilde g_{2^t-4}\widetilde a_{2^t-4}$. As the first step in that direction, we prove the following proposition.

\begin{proposition}\label{lambde}
If $\lambda_k$, $0\le k\le\lfloor(2^{t-1}-2)/3\rfloor$, are the coefficients from (\ref{a^2_preko_baznih}), then
\[\sum_{k=0}^{\left\lfloor \frac{2^{t-1}-2}{3} \right\rfloor} \lambda_k \widetilde a_{2^t-4} \widetilde w_2^{2^{t-1}-2-3k}\widetilde w_3^{2k}=\lambda_0\widetilde g_{2^t-4}\widetilde a_{2^t-4}.\]
\end{proposition}
\begin{proof}
The main result of \cite{CP} is the equality $\widetilde a_{2^t-1}^2=0$, where $\widetilde a_{2^t-1}\in H^{2^t-1}(\widetilde G_{2^t,3})$ is the indecomposable class. Therefore, by (\ref{i*2^t-->2^t-1}) we get
\[\widetilde w_3^2\widetilde a_{2^t-4}^2=(\widetilde w_3\widetilde a_{2^t-4})^2=\big(\,\widetilde i^*(\widetilde a_{2^t-1})\big)^2=\widetilde i^*(\widetilde a_{2^t-1}^2)=0.\]
So, multiplying (\ref{a^2_preko_baznih}) by $\widetilde w_3^2$ leads to 
  \begin{equation}\label{w3^2a^2=0}
        0 = \widetilde w_3^2 \widetilde a_{2^t-4}^2 = \sum_{k=0}^{\left\lfloor \frac{2^{t-1}-2}{3} \right\rfloor} \lambda_k \widetilde a_{2^t-4} \widetilde w_2^{2^{t-1}-2-3k}\widetilde w_3^{2k+2}
    \end{equation}
($\gamma\widetilde w_2^{2^t-4}\widetilde w_3^2=0$ due to (\ref{propimp*})). The classes $\widetilde a_{2^t-4} \widetilde w_2^{2^{t-1}-2-3k}\widetilde w_3^{2k+2}$ for $k\ge1$ are elements of the basis $B$, since there is no $i\in\{0,1,\ldots,t-1\}$ such that $2^{t-1}-2-3k\ge2^{t-1}-2^i$ and $2k+2\ge2^i-1$ (if $k\ge1$). However, for $k=0$ we have the class $\widetilde a_{2^t-4} \widetilde w_2^{2^{t-1}-2}\widetilde w_3^2$, which is not in $B$. We wish to express this class as a sum of basis elements. In order to do so, note that the monomial $a_{2^t-4}w_2^{2^{t-1}-2}w_3^2$ is divisible by $\mathrm{LM}(f_1)=w_2^{2^{t-1}-2}w_3$ (see (\ref{LMfi})), and by (\ref{grebnerova_baza}) we have
\begin{equation}\label{f_1} 
   f_1 =g_{2^t-1} = \sum_{2b+3c=2^t-1} \binom{b+c}{b}w_2^b w_3^c= \sum_{k = 0}^{\left\lfloor \frac{2^{t-1}-2}{3} \right\rfloor} \binom{2^{t-1}-1-k}{2k+1} w_2^{2^{t-1}-2-3k}w_3^{2k+1}. 
\end{equation}
The last equality comes from the assertion that $c$ must be odd (since $2b+3c=2^t-1$), and then one substitutes $2k+1$ for $c$. We know that $\widetilde f_1=0$ in $H^*(\widetilde G_{2^t-1,3})$, and so
\begin{align*}
0&=\sum_{k = 0}^{\left\lfloor \frac{2^{t-1}-2}{3} \right\rfloor} \binom{2^{t-1}-1-k}{2k+1}\widetilde w_2^{2^{t-1}-2-3k}\widetilde w_3^{2k+1}\\
 &=\widetilde w_2^{2^{t-1}-2}\widetilde w_3+\sum_{k = 1}^{\left\lfloor \frac{2^{t-1}-2}{3} \right\rfloor} \binom{2^{t-1}-1-k}{2k+1}\widetilde w_2^{2^{t-1}-2-3k}\widetilde w_3^{2k+1}.
\end{align*}
Multiplying this equation with $\widetilde a_{2^t-4}\widetilde w_3$ we get
    \begin{equation}\label{preko_baznih}
    \widetilde a_{2^t-4}\widetilde w_2^{2^{t-1}-2} \widetilde w_3^2 = \sum_{k = 1}^{\left\lfloor \frac{2^{t-1}-2}{3} \right\rfloor} \binom{2^{t-1}-1-k}{2k+1} \widetilde a_{2^t-4}\widetilde w_2^{2^{t-1}-2-3k} \widetilde w_3^{2k+2}
    \end{equation}
Inserting (\ref{preko_baznih}) in (\ref{w3^2a^2=0}) leads to
\[0 = \sum_{k=1}^{\left\lfloor \frac{2^{t-1}-2}{3} \right\rfloor} \left(\lambda_k + \lambda_0 \binom{2^{t-1} -1 -k}{2k+1}\right) \widetilde a_{2^t-4} \widetilde w_2^{2^{t-1}-2-3k}\widetilde w_3^{2k+2}.\] 
As we have already indicated, the right-hand side is a linear combination of \em basis \em elements, which means that all coefficients must vanish. Therefore, $\lambda_k = \lambda_0 \binom{2^{t-1} -1 -k}{2k+1}$ for all $k$ such that $1\le k\le\lfloor(2^{t-1}-2)/3\rfloor$. We conclude that 
\[\sum_{k=0}^{\left\lfloor \frac{2^{t-1}-2}{3} \right\rfloor} \lambda_k \widetilde a_{2^t-4} \widetilde w_2^{2^{t-1}-2-3k}\widetilde w_3^{2k}=\lambda_0\widetilde a_{2^t-4} \sum_{k=0}^{\left\lfloor \frac{2^{t-1}-2}{3} \right\rfloor} \binom{2^{t-1} -1 -k}{2k+1} \widetilde w_2^{2^{t-1}-2-3k}\widetilde w_3^{2k}.\]
It remains to prove that this sum on the right-hand side is equal to $\widetilde g_{2^t-4}$. If
\[S = \sum_{k=0}^{\left\lfloor \frac{2^{t-1}-2}{3} \right\rfloor} \binom{2^{t-1} -1 -k}{2k+1} w_2^{2^{t-1}-2-3k}  w_3^{2k},\] 
then using (\ref{f_1}) and (\ref{recgpolk3}) we get that in $\mathbb Z_2[w_2,w_3]$ the following identity holds
\[w_3 S =  g_{2^t-1} = w_2 g_{2^t-3} + w_3 g_{2^t-4} = w_3 g_{2^t-4}\]
(since $g_{2^t-3} = 0$). Canceling out $w_3$ finishes the proof.
\end{proof}

So now (\ref{a^2_preko_baznih}) simplifies to
\begin{equation}\label{eqq}
\widetilde a_{2^t-4}^2=\lambda_0\widetilde g_{2^t-4}\widetilde a_{2^t-4}+\gamma\widetilde w_2^{2^t-4}
\end{equation}
(for some $\lambda_0,\gamma\in\mathbb Z_2$), and in order to obtain (\ref{eq}) (and thus finish the proof of Theorem \ref{mainthm}(a)) we are left to prove that $\lambda_0=1$.

\begin{proposition}\label{kercapker1}
    Let $H^{2^t-4}(G_{2^t-1,3})\stackrel{w_1}\longrightarrow H^{2^t-3}(G_{2^t-1,3})$ be the multiplication map and $j^*:H^{2^t-4}(G_{2^t-1,3})\rightarrow H^{2^t-4}(G_{2^t-2,2})$ the map induced by the embedding $j:G_{2^t-2,2}\hookrightarrow G_{2^t-1,3}$ (see Section \ref{preliminaries}). Then
    $\ker w_1\cap \ker j^* = 0$.
\end{proposition}
\begin{proof} As in the proof of Proposition \ref{kercapker2}, let $f\in\mathbb Z_2[w_1,w_2,w_3]$ be such that for the corresponding cohomology class in $H^{2^t-4}(G_{2^t-1,3})$ one has $[f]\in \ker w_1\cap \ker j^*$. We want to prove that $[f]=0$ in $H^{2^t-4}(G_{2^t-1,3})$.

By Lemma \ref{naturality}(b), $j^*[f] = [\rho(f)]$, where $\rho:\mathbb Z_2[w_1,w_2,w_3]\rightarrow\mathbb Z_2[w_1,w_2]$ is the reduction modulo $w_3$. Since $j^*[f] = 0$, we first get $[\rho(f)]=0$, but actually $\rho(f)=0$, because the degree of $\rho(f)$ in $\mathbb Z_2[w_1,w_2]$ is $2^t-4$, and $H^*(G_{2^t-2,2})\cong\mathbb Z_2[w_1,w_2]/(\overline w_{2^t-3},\overline w_{2^t-2})$. We conclude that $f=w_3h$, for some $h\in\mathbb Z_2[w_1,w_2,w_3]$.

On the other hand, $[f]\in \ker w_1$ means that $[w_1f] = 0$ in $H^{2^t-3}(G_{2^t-1,3})$, and since $H^*(G_{2^t-1,3})\cong\mathbb Z_2[w_1,w_2,w_3]/(\overline w_{2^t-3},\overline w_{2^t-2},\overline w_{2^t-1})$, we conclude that $w_1f = \alpha \overline w_{2^t-3}$ for some $\alpha\in\mathbb Z_2$. When we combine these two results, we get 
\[w_1w_3h = \alpha \overline w_{2^t-3}.\]
If $\alpha$  were nonzero, then the polynomial $\overline w_{2^t-3}$ would be divisible by $w_3$. However, it is obvious that the monomial $w_1^{2^t-3}$ appears in $\overline w_{2^t-3}$ with nonzero coefficient (see (\ref{overline_w})). Therefore, $\alpha = 0$, which implies $w_1f = \alpha \overline w_{2^t-3}=0$, and so $f=0$.
\end{proof}

We know that $\widetilde a_{2^t-4}\in H^{2^t-4}(\widetilde G_{2^t-1,3})\setminus\im p^*$, and so Proposition \ref{kercapker1} and Lemma \ref{osnovna_lema}(b) (for $r=2^t-4$, $n=2^t-2$ and $k=2$) imply 
\begin{equation}\label{j*a}
  \widetilde j^*(\widetilde a_{2^t-4})\in H^{2^t-4}(\widetilde G_{2^t-2,2})\setminus\im p^*  
\end{equation} 
(here, $\im p^*$ is the subalgebra of $H^*(\widetilde G_{2^t-2,2})$ generated by $\widetilde w_2$). A complete des\-cription of $H^*(\widetilde G_{2^t-2,2})$ is given in \cite[Theorem 2.1(b)]{KorbasRusin:Palermo} and reads as follows:
\begin{equation}\label{Korbas_n,2}
    H^*(\widetilde G_{2^t-2,2})\cong\frac{\mathbb Z_2[w_2,b_{2^t-4}]}{\big(w_2^{2^{t-1}-1},b_{2^t-4}^2 - w_2^{2^{t-1}-2}b_{2^t-4}\big)}.
    \end{equation}
If $\widetilde b_{2^t-4}\in H^{2^t-4}(\widetilde G_{2^t-2,2})$ is the class corresponding to (the coset of) $b_{2^t-4}$ via the isomorphism (\ref{Korbas_n,2}), then (\ref{j*a}) implies
\begin{equation}\label{j*a=}
 \widetilde j^*(\widetilde a_{2^t-4})=\widetilde b_{2^t-4}+\mu\widetilde w_2^{2^{t-1}-2},   
\end{equation}
for some $\mu\in\mathbb Z_2$. Since $\widetilde w_2^{2^t-4}=0$ in $H^*(\widetilde G_{2^t-2,2})$, squaring the equation (\ref{j*a=}) leads to
\begin{equation}\label{j*a^2}
 \widetilde j^*(\widetilde a_{2^t-4}^2)=\widetilde b_{2^t-4}^2\neq0 \quad \mbox{ in } H^{2^{t+1}-8}(\widetilde G_{2^t-2,2}). 
\end{equation}
On the other hand, by (\ref{eqq}) and (\ref{naturalityj})
\[\widetilde j^*(\widetilde a_{2^t-4}^2)=\lambda_0\widetilde j^*(\widetilde g_{2^t-4}\widetilde a_{2^t-4})+\gamma\widetilde w_2^{2^t-4}=\lambda_0\widetilde j^*(\widetilde g_{2^t-4}\widetilde a_{2^t-4}),\]
leading to the conclusion $\lambda_0=1$. This completes the proof of Theorem \ref{mainthm}(a).

\medskip

Note that we have proved that $\widetilde a_{2^t-4}^2\neq0$ in $H^{2^{t+1}-8}(\widetilde G_{2^t-1,3})$. This follows from (\ref{j*a^2}), and it also follows from the description of $H^*(\widetilde G_{2^t-1,3})$ (given in Theorem \ref{mainthm}(a)) itself.

We were not able to compute the coefficient $\gamma\in\mathbb Z_2$ from this description, and likewise we could not deduce whether the top dimensional cohomology class $\widetilde a_{2^t-4}^3\in H^{3\cdot2^t-12}(\widetilde G_{2^t-1,3})$ vanishes or not. Nevertheless, in the upcoming Proposition \ref{ekv} we prove that these two problems are equivalent.

We first show that $\widetilde w_2^{2^{t-1}-2}$ is the only basis element in $\im p^*\cap H^{2^t-4}(\widetilde G_{2^t-1,3})$ whose square is nonzero. Actually, we already know that $(\widetilde w_2^{2^{t-1}-2})^2=\widetilde w_2^{2^t-4}\neq0$ (see (\ref{visina_w_2})), so we now verify that the square of $\widetilde w_2^{2^{t-1}-2-3k}\widetilde w_3^{2k}$ vanishes if $k>0$.

\begin{lemma}\label{kvadrati_su_nula}
If $k>0$ and $2^t-4-6k\ge0$, then
\[\widetilde w_2^{2^t-4-6k}\widetilde w_3^{4k}=0 \quad \mbox{ in } H^*(\widetilde G_{2^t-1,3}).\]
\end{lemma}
\begin{proof}
The proof is by reverse induction on $k$. If $k\ge2^{t-3}$, i.e., $4k\ge2^{t-1}$, then
\[\widetilde w_2^{2^t-4-6k}\widetilde w_3^{4k}=\widetilde w_2^{2^t-4-6k}\widetilde w_3^{4k-2^{t-1}+1}\widetilde w_3^{2^{t-1}-1}=\widetilde w_2^{2^t-4-6k}\widetilde w_3^{4k-2^{t-1}+1}\widetilde f_{t-1}=0,\]
by (\ref{ft-1}).

Now let $0<k<2^{t-3}$ and suppose that $\widetilde w_2^{2^t-4-6j}\widetilde w_3^{4j}=0$ for all integers $j>k$ (such that $2^t-4-6j\ge0$). There exists a unique integer $i\in\{2,\ldots,t-2\}$ with the property $2^i\le4k\le2^{i+1}-4$. Note that then 
\[12k\le3\cdot2^{i+1}-12=2^{i+2}+2^{i+1}-12\le2^t+2^{i+1}-12,\]
i.e., $6k\le2^{t-1}+2^i-6$, implying $2^t-4-6k>2^{t-1}-2^i$. This means that we can write
\begin{equation}\label{eq1}
\widetilde w_2^{2^t-4-6k}\widetilde w_3^{4k}=\widetilde w_2^{2^t-4-6k-2^{t-1}+2^i}\widetilde w_3^{4k-2^i+1}\widetilde w_2^{2^{t-1}-2^i}\widetilde w_3^{2^i-1}.
\end{equation}
On the other hand, according to (\ref{LMfi}) and (\ref{grebnerova_baza}), 
\[w_2^{2^{t-1}-2^i}w_3^{2^i-1}=\mathrm{LM}(f_i)=f_i+\!\!\!\sum_{\substack{2b+3c=2^t-3+2^i\\(b,c)\neq(2^{t-1}-2^i,2^i-1)}}{b+c\choose b}w_2^bw_3^c.\]
Since $\widetilde f_i=0$, by (\ref{eq1}) in cohomology we have:
\begin{align*}
\widetilde w_2^{2^t-4-6k}\widetilde w_3^{4k}&=\widetilde w_2^{2^{t-1}+2^i-4-6k}\widetilde w_3^{4k-2^i+1}\widetilde w_2^{2^{t-1}-2^i}\widetilde w_3^{2^i-1}\\
                                            &=\widetilde w_2^{2^{t-1}+2^i-4-6k}\widetilde w_3^{4k-2^i+1}\!\!\!\sum_{\substack{2b+3c=2^t-3+2^i\\(b,c)\neq(2^{t-1}-2^i,2^i-1)}}{b+c\choose b}\widetilde w_2^b\widetilde w_3^c\\
                                            &=\!\!\!\sum_{\substack{2b+3c=2^t-3+2^i\\(b,c)\neq(2^{t-1}-2^i,2^i-1)}}{b+c\choose b}\widetilde w_2^{b+2^{t-1}+2^i-4-6k}\widetilde w_3^{c+4k-2^i+1}.\\
\end{align*}
To complete the induction step, it now suffices to prove that for a summand with ${b+c\choose b}\equiv1\pmod2$ one has $c+4k-2^i+1=4j$ for some $j>k$.

First, since $\mathrm{LM}(f_i)=w_2^{2^{t-1}-2^i}w_3^{2^i-1}$, for a nonzero summand in this sum, one must have $b<2^{t-1}-2^i$, and consequently $c>2^i-1$ (because $2b+3c$ is the same for all monomials in $f_i$). This implies $c+4k-2^i+1>4k$, and so we only need to prove that $c+4k-2^i+1$ is divisible by $4$.

Since $i\ge2$ and $2b+3c=2^t-3+2^i$, we first conclude that $c$ is odd, and therefore $c+4k-2^i+1$ is even, which means that it remains to rule out the possibility $c+4k-2^i+1\equiv2\pmod4$, i.e., $c\equiv1\pmod4$. But $c\equiv1\pmod4$ and $2b+3c=2^t-3+2^i$ would imply $2b\equiv2\pmod4$, and we would have that $b$ is odd, contradicting the fact ${b+c\choose b}\equiv1\pmod2$.
\end{proof}

This lemma ensures that
\begin{equation}\label{g_na_kvadrat}
\widetilde g_{2^t-4}^2=\widetilde w_2^{2^t-4}.
\end{equation}
Namely, the leading monomial of $g_{2^t-4}$ is obviously $w_2^{2^{t-1}-2}$, and all other monomials are of the form $w_2^{2^{t-1}-2-3k}w_3^{2k}$ for some $k>0$.

As we have already indicated, the indecomposable class in $H^{2^t-4}(\widetilde G_{2^t-1,3})$ is not unique. If $\widetilde a_{2^t-4}$ is such a class, then so is $\widetilde a_{2^t-4}+\widetilde w$ for any $\widetilde w\in\im p^*\cap H^{2^t-4}(\widetilde G_{2^t-1,3})$ ($\widetilde a_{2^t-4}$ is unique only up to addition of such a $\widetilde w$). As another consequence of Lemma \ref{kvadrati_su_nula}, we note that taking some other indecomposable class from $H^{2^t-4}(\widetilde G_{2^t-1,3})$ has no effect on equality (\ref{eq}). Indeed, 
\[(\widetilde a_{2^t-4}+\widetilde w)^2=\widetilde a_{2^t-4}^2+\widetilde w^2=\widetilde g_{2^t-4}\widetilde a_{2^t-4}+\gamma\widetilde w_2^{2^t-4}+\widetilde w^2=\widetilde g_{2^t-4}(\widetilde a_{2^t-4}+\widetilde w)+\gamma\widetilde w_2^{2^t-4},\]
where the last equality comes from Lemma \ref{kvadrati_su_nula}. Namely, one can fairly easily show that $\widetilde w^2=\widetilde w\widetilde g_{2^t-4}$ using the lemma and the following facts: 
\begin{itemize}
    \item $\widetilde w$ is a sum of monomials of the form $\widetilde w_2^{2^{t-1}-2-3k}\widetilde w_3^{2k}$ for some $k\ge0$; 
    \item $\widetilde w_3\widetilde g_{2^t-4}=\widetilde g_{2^t-1}=0$ in $H^*(\widetilde G_{2^t-1,3})$ (see (\ref{r>=n-2})); 
    \item for every monomial $\widetilde w_2^b\widetilde w_3^c$ with nonzero coefficient in $\widetilde g_{2^t-4}$ the exponent $c$ is divisible by $4$ (see the proof of Proposition \ref{lambde}, where $g_{2^t-4}$ is denoted by $S$).

\end{itemize}

\begin{proposition}\label{ekv}
For an indecomposable class $\widetilde a_{2^t-4}\in H^{2^t-4}(\widetilde G_{2^t-1,3})$ and the coefficient $\gamma\in\mathbb Z_2$ from Theorem \ref{mainthm}(a) the following equivalence holds:
\[\widetilde a_{2^t-4}^3\neq0 \quad \Longleftrightarrow \quad \gamma=0.\]
\end{proposition}
\begin{proof}
By (\ref{eq}) and (\ref{g_na_kvadrat}) we have
\begin{align*}
\widetilde a_{2^t-4}^3&=\widetilde g_{2^t-4}\widetilde a_{2^t-4}^2+\gamma\widetilde a_{2^t-4}\widetilde w_2^{2^t-4}=\widetilde g_{2^t-4}^2\widetilde a_{2^t-4}+\gamma\widetilde g_{2^t-4}\widetilde w_2^{2^t-4}+\gamma\widetilde a_{2^t-4}\widetilde w_2^{2^t-4}\\
                      &=\widetilde g_{2^t-4}^2\widetilde a_{2^t-4}+\gamma\widetilde a_{2^t-4}\widetilde w_2^{2^t-4}=(1+\gamma)\widetilde a_{2^t-4}\widetilde w_2^{2^t-4},\\
\end{align*}
since $\widetilde g_{2^t-4}\widetilde w_2^{2^t-4}\in H^{3\cdot2^t-12}(\widetilde G_{2^t-1,3})\cap\im p^*=0$ (by (\ref{propimp*})). On the other hand, by (\ref{topdim}) we know that $\widetilde a_{2^t-4}\widetilde w_2^{2^t-4}$ is the nonzero element in $H^{3\cdot2^t-12}(\widetilde G_{2^t-1,3})$, and so this equality proves the proposition.
\end{proof}

\section{The case $n=2^t-2$}
\label{2t-2}

By \cite[Theorem A(b)]{BasuChakraborty} the cohomology algebra $H^*(\widetilde G_{2^t-2,3})$ contains an indecomposable class in dimension $2^t-4$. Similarly as in the case $n=2^t-1$, for this class we can choose any $\widetilde a_{2^t-4}\in H^{2^t-4}(\widetilde G_{2^t-2,3})\setminus\im p^*$. If $\widetilde i:\widetilde G_{2^t-2,3}\hookrightarrow\widetilde G_{2^t-1,3}$ is the embedding and $\widetilde a_{2^t-4}\in H^{2^t-4}(\widetilde G_{2^t-1,3})$ the indecomposable class from the previous section, we will prove that $\widetilde i^*(\widetilde a_{2^t-4})\in H^{2^t-4}(\widetilde G_{2^t-2,3})\setminus\im p^*$, and then make a choice:
\begin{equation}\label{a=i*a}
  \widetilde a_{2^t-4}:=\widetilde i^*(\widetilde a_{2^t-4}). 
\end{equation}
Similarly as before, the relation $\widetilde i^*(\widetilde a_{2^t-4})\in H^{2^t-4}(\widetilde G_{2^t-2,3})\setminus\im p^*$ will be a direct consequence of Lemma \ref{osnovna_lema}(a) and the following proposition.

\begin{proposition}\label{kercapker3}
   Let $H^{2^t-4}(G_{2^t-1,3})\stackrel{w_1}\longrightarrow H^{2^t-3}(G_{2^t-1,3})$ be the multiplication map and $i^*:H^{2^t-4}(G_{2^t-1,3})\rightarrow H^{2^t-4}(G_{2^t-2,3})$ the map induced by the embedding $i:G_{2^t-2,3}\hookrightarrow G_{2^t-1,3}$ (see Section \ref{preliminaries}). Then
    $\ker w_1\cap \ker i^* = 0$.
\end{proposition}
\begin{proof}
The proof is similar to those of Propositions \ref{kercapker2} and \ref{kercapker1}. Take a polynomial $f\in\mathbb Z_2[w_1,w_2,w_3]$ such that $[f]\in \ker w_1\cap \ker i^*$. By Lemma \ref{naturality}(a), $[f]=i^*[f]=0$ in $H^{2^t-4}(G_{2^t-2,3})$, and so $f = \alpha \overline w_{2^t-4}$ for some $\alpha\in\mathbb Z_2$. On the other hand, since $[w_1f]=0$ in $H^{2^t-3}(G_{2^t-1,3})$, we have $w_1f = \beta \overline w_{2^t-3}$ for some $\beta\in\mathbb Z_2$. Combining these two equations, we get 
\begin{equation}\label{eq3}
    \alpha w_1\overline w_{2^t-4} = \beta \overline w_{2^t-3}.
\end{equation}
If $\alpha$ were nonzero, then the monomial $w_1^3w_2^{2^{t-1}-3}$ would appear on the left-hand side of (\ref{eq3}) with nonzero coefficient (since $w_1^2w_2^{2^{t-1}-3}$ appears in $\overline w_{2^t-4}$ with nonzero coefficient, by (\ref{overline_w})). However, the coefficient of $w_1^3w_2^{2^{t-1}-3}$ in $\overline w_{2^t-3}$ is $\binom{2^{t-1}}{3}=0$. We conclude that $\alpha=0$, and so $f=\alpha\overline w_{2^t-4}=0$.
\end{proof}

So now we have the indecomposable class $\widetilde a_{2^t-4}\in H^{2^t-4}(\widetilde G_{2^t-2,3})$, chosen by (\ref{a=i*a}). According to \cite[Theorem B(2)]{BasuChakraborty} there exists an algebra isomorphism
\begin{equation}\label{isomorphism2^t-2}
 H^*(\widetilde G_{2^t-2,3})\cong\frac{\mathbb Z_2[w_2,w_3,a_{2^t-4}]}{\big(g_{2^t-4},g_{2^t-2},a_{2^t-4}^2-P_1a_{2^t-4}\big)},  
 \end{equation}
 which maps the classes $\widetilde w_2$, $\widetilde w_3$ and $\widetilde a_{2^t-4}$ to the cosets of $w_2$, $w_3$ and $a_{2^t-4}$ respectively, and $P_1$ is some polynomial in variables $w_2$ and $w_3$ only. In order to prove part (b) of Theorem \ref{mainthm}, we need to show that for $P_1$ we can take the zero polynomial, which amounts to showing that $\widetilde a_{2^t-4}^2=0$ in $H^*(\widetilde G_{2^t-2,3})$. We shall need the following lemma.

 \begin{lemma}\label{w_2}
  We have $\widetilde w_2^{2^t-4}=0$ in $H^{2^{t+1}-8}(\widetilde G_{2^t-2,3})$.
 \end{lemma}
\begin{proof}
    Assume to the contrary that $\widetilde w_2^{2^t-4}\neq0$. The dimension of the manifold $\widetilde G_{2^t-2,3}$ is $3\cdot2^t-15$, and so by the Poincar\'e duality, there exists a class $\sigma\in H^{3\cdot2^t-15-(2^{t+1}-8)}(\widetilde G_{2^t-2,3})$ such that $\sigma\widetilde w_2^{2^t-4}\neq0$ in the top dimensional cohomology group $H^{3\cdot2^t-15}(\widetilde G_{2^t-2,3})$. Since $3\cdot2^t-15-(2^{t+1}-8)=2^t-7<2^t-4$, $\sigma$ is a polynomial in $\widetilde w_2$ and $\widetilde w_3$, i.e., $\sigma\in\im p^*$. This means that the nonzero top dimensional class $\sigma\widetilde w_2^{2^t-4}$ belongs to $\im p^*$. However, the relevant part of the Gysin sequence shows that this is not the case.
    \[
\begin{tikzcd}[column sep = 1.3em]
    \cdots
    \arrow{r}
    &
    H^{3\cdot 2^t - 15}(G_{2^t-2,3})
    \arrow{r}{p^*}
    &
    H^{3\cdot 2^t - 15}(\widetilde G_{2^t-2,3})
    \arrow{r}{\varphi}
    &
    H^{3\cdot 2^t - 15}(G_{2^t-2,3})
    \arrow{r}
    &
    0
\end{tikzcd}
\]
Namely, $H^{3\cdot 2^t - 15}(\widetilde G_{2^t-2,3})\cong H^{3\cdot 2^t - 15}(G_{2^t-2,3})\cong\mathbb Z_2$, $\varphi$ is onto, and hence isomorphism. Therefore, $p^*=0$. This contradiction concludes the proof.
\end{proof}

Now, according to (\ref{a=i*a}), (\ref{eq}) and (\ref{naturalityi})
\[\widetilde a_{2^t-4}^2=\widetilde i^*(\widetilde a_{2^t-4}^2)=\widetilde i^*(\widetilde g_{2^t-4})\cdot\widetilde i^*(\widetilde a_{2^t-4})+\gamma\widetilde i^*(\widetilde w_2^{2^t-4})=\widetilde g_{2^t-4}\widetilde a_{2^t-4}+\gamma\widetilde w_2^{2^t-4}.\]
By Lemma \ref{w_2}, $\widetilde w_2^{2^t-4}=0$, and by (\ref{r>=n-2}), $\widetilde g_{2^t-4}=0$ in $H^*(\widetilde G_{2^t-2,3})$. So, $\widetilde a_{2^t-4}^2=0$, and the proof of Theorem \ref{mainthm}(b) is complete.

\section{The case $n=2^t-3$}
\label{2t-3}

In this section we prove part (c) of Theorem \ref{mainthm}, and the proof goes along the same lines as the proof of part (b) (given in the previous section). We choose the indecomposable class $\widetilde a_{2^t-4}\in H^{2^t-4}(\widetilde G_{2^t-3,3})$ to be $\widetilde i^*(\widetilde a_{2^t-4})$, where $\widetilde a_{2^t-4}\in H^{2^t-4}(\widetilde G_{2^t-2,3})$ is the indecomposable class chosen in the previous section, and $\widetilde i:\widetilde G_{2^t-3,3}\hookrightarrow\widetilde G_{2^t-2,3}$ is the embedding from Section \ref{preliminaries}. When we do so, then it is immediate that $\widetilde a_{2^t-4}^2=\widetilde i^*(\widetilde a_{2^t-4}^2)=0$, and Theorem \ref{mainthm}(c) follows.

So the only thing to be proved is the claim that we can take $\widetilde i^*(\widetilde a_{2^t-4})$ as the indecomposable class in $H^{2^t-4}(\widetilde G_{2^t-3,3})$, that is, $\widetilde i^*(\widetilde a_{2^t-4})\in H^{2^t-4}(\widetilde G_{2^t-3,3})\setminus\im p^*$. As in the previous cases, Lemma \ref{osnovna_lema}(a) ensures that it is enough to establish the following proposition.

\begin{proposition}
     For the multiplication map $H^{2^t-4}(G_{2^t-2,3})\stackrel{w_1}\longrightarrow H^{2^t-3}(G_{2^t-2,3})$ and the map $i^*:H^{2^t-4}(G_{2^t-2,3})\rightarrow H^{2^t-4}(G_{2^t-3,3})$ induced by the embedding $i:G_{2^t-3,3}\hookrightarrow G_{2^t-2,3}$, we have
    $\ker w_1\cap \ker i^* = 0$.
\end{proposition}
\begin{proof}
    For a polynomial $f\in\mathbb Z_2[w_1,w_2,w_3]$ such that $[f]\in \ker w_1\cap \ker i^*$ we have $[f]=i^*[f]=0$ in $H^{2^t-4}(G_{2^t-3,3})$, and so $f = \alpha w_1 \overline w_{2^t-5} + \beta \overline w_{2^t-4}$ for some $\alpha,\beta\in\mathbb Z_2$. But $\overline w_{2^t-4}\in J_{2^t-2,3}$, which means that in $H^{2^t-4}(G_{2^t-2,3})$
    \[[f] = [\alpha w_1 \overline w_{2^t-5} + \beta \overline w_{2^t-4}] = [\alpha w_1 \overline w_{2^t-5}],\]
    and so we can take $\beta=0$. We wish to prove that $\alpha=0$ as well.

We also know that $[w_1f] = 0$ in $H^{2^t-3}(G_{2^t-2,3})$, and conclude $w_1f = \lambda w_1 \overline w_{2^t-4} + \mu \overline w_{2^t-3}$, for some $\lambda,\mu\in\mathbb Z_2$. Combining these equations, we obtain
\begin{equation}\label{eq4}
    \alpha w_1^2 \overline w_{2^t-5} = \lambda w_1 \overline w_{2^t-4} + \mu \overline w_{2^t-3}.
\end{equation}

Observe the monomial $w_1^{2^t-5}w_2$. We use (\ref{overline_w}) and in Table \ref{tab:4} list its coefficients in each of the summands from (\ref{eq4}).
\begin{table}[h]
\renewcommand{\arraystretch}{1.5} %
\begin{tabular}{|c|c|c|c|c|}
 \hline
 summand & $a$ & $b$ & $c$ & $\binom{a+b+c}{a}\binom{b+c}{b}$ \\ [1ex]
 \hline
$\alpha w_1^2 \overline w_{2^t-5}$ & $2^t-7$ & 1 & 0 &   0\\ 
$\lambda w_1 \overline w_{2^t-4}$ & $2^t-6$ & 1 & 0 &   1\\
$\mu \overline w_{2^t-3}$ & $2^t-5$ & 1 & 0 &   0\\
 \hline
  \end{tabular}
\caption{\label{tab:4} Coefficients of $w_1^{2^t-5}w_2$ in (\ref{eq4})}
\end{table}

From the table we conclude that $\lambda = 0$. This means equation (\ref{eq4}) now becomes
\begin{equation}\label{eq5}
    \alpha w_1^2 \overline w_{2^t-5} = \mu \overline w_{2^t-3}.
\end{equation}
The monomial $w_1w_2^{2^{t-1}-2}$ appears on the right-hand side of (\ref{eq5}) with the coefficient $\mu\binom{2^{t-1}-1}{1}\binom{2^{t-1}-2}{0} = \mu$, and it is obvious that it cannot appear on the left-hand side. We conclude that $\mu=0$, and consequently $\alpha=0$. This finishes the proof.
\end{proof}

\bibliographystyle{amsplain}

\end{document}